\documentclass{amsart}

\usepackage{amsmath}
\usepackage{amsthm}
\usepackage{amsfonts}
\usepackage{amssymb}

\newcommand\R{{\mathbb{R}}}
\newcommand\C{{\mathbb{C}}}

\renewcommand\P{{\mathbf{P}}}
\newcommand\E{{\mathbf{E}}}

\renewcommand\Im{{\operatorname{Im}}}
\renewcommand\Re{{\operatorname{Re}}}
\newcommand\eps{{\varepsilon}}

\newcommand\Dyson{{\operatorname{Sine}}}

\newcommand\condo{{{\bf C0}}}

\subjclass{15A52}

\newcommand\cl{{\operatorname{cl}}}
\renewcommand\sc{{\operatorname{sc}}}

\theoremstyle{plain}
  \newtheorem{theorem}{Theorem}

  \newtheorem{lemma}[theorem]{Lemma}
  \newtheorem{corollary}[theorem]{Corollary}

\theoremstyle{definition}
  \newtheorem{definition}[theorem]{Definition}
  \newtheorem{example}[theorem]{Example}
  \newtheorem{remark}[theorem]{Remark}

\include{psfig}

\begin{document}

\title[The Four Moment Theorem]{Random matrices:\\  The Four Moment Theorem for Wigner ensembles}

\author{Terence Tao}
\address{Department of Mathematics, UCLA, Los Angeles CA 90095-1555}
\email{tao@math.ucla.edu}
\thanks{T. Tao is supported by a grant from the MacArthur Foundation, by NSF grant DMS-0649473, and by the NSF Waterman award.}

\author{Van Vu}
\address{Department of Mathematics, Rutgers, Piscataway, NJ 08854}
\email{vanvu@math.rutgers.edu}
\thanks{V. Vu is supported by research grants DMS-0901216 and AFOSAR-FA-9550-09-1-0167.}

\begin{abstract}  We survey some recent progress on rigorously establishing the universality of various spectral statistics of Wigner random matrix ensembles, focusing in particular on the Four Moment Theorem and its 
applications.  
\end{abstract}

\maketitle

\setcounter{tocdepth}{2}

\section{Introduction}


The purpose of this paper is to survey the \emph{Four Moment Theorem} and its applications in understanding the asymptotic spectral properties of random matrix ensembles of Wigner type.  Due to limitations of space, this survey will be far from exhaustive; an extended version of this survey will appear elsewhere.  (See also \cite{Erd}, \cite{Alice}, \cite{schlein} for some recent surveys in this area.)

To simplify the exposition (at the expense of stating the results in maximum generality), we shall restrict attention to a model class of random matrix ensembles, in which we assume somewhat more decay and identical distribution hypotheses than are strictly necessary for the main results.

\begin{definition}[Wigner matrices]\label{def:Wignermatrix}  Let $n \geq 1$ be an integer (which we view as a parameter going off to infinity). An $n \times n$ \emph{Wigner Hermitian matrix} $M_n$ is defined to be a  random Hermitian $n \times n$ matrix $M_n = (\xi_{ij})_{1 \leq i,j \leq n}$, in which the $\xi_{ij}$ for $1 \leq i \leq j \leq n$ are jointly independent with $\xi_{ji} = \overline{\xi_{ij}}$ (in particular, the $\xi_{ii}$ are real-valued).  For $1 \leq i < j \leq n$, we require that the $\xi_{ij}$ have mean zero and variance one, while for $1 \leq i=j \leq n$ we require that the $\xi_{ij}$ have mean zero and variance $\sigma^2$ for some $\sigma^2>0$ independent of $i,j,n$.   To simplify some of the statements of the results here, we will also assume that the $\xi_{ij} \equiv \xi$ are identically distributed for $i < j$, and the $\xi_{ii} \equiv \xi'$ are also identically distributed for $i=j$, and furthermore that the real and imaginary parts of $\xi$ are independent.  We refer to the distributions $\Re \xi$, $\Im \xi$, and $\xi'$ as the \emph{atom distributions} of $M_n$.

We say that the Wigner matrix ensemble \emph{obeys Condition {\condo}} if we have the exponential decay condition
\begin{equation*}
\P(|\xi_{ij}|\ge t^C) \le e^{-t} 
\end{equation*}
for all $1 \leq i,j \leq n$ and $t \ge C'$, and some constants $C, C'$ (independent of $i,j,n$).  
\end{definition}

We refer to the matrix $W_n := \frac{1}{\sqrt{n}} M_n$ as the \emph{coarse-scale normalised Wigner Hermitian matrix}, and $A_n := \sqrt{n} M_n$ as the \emph{fine-scale normalised Wigner Hermitian matrix}.

\begin{example}[Invariant ensembles]\label{gue}  An important special case of a Wigner Hermitian matrix $M_n$ is the \emph{gaussian unitary ensemble} (GUE), in which $\xi_{ij} \equiv N(0,1)_\C$ are complex gaussians with mean zero and variance one for $i \neq j$, and $\xi_{ii} \equiv N(0,1)_\R$ are real gaussians with mean zero and variance one for $1 \leq i \leq n$ (thus $\sigma^2 = 1$ in this case).  Another important special case is the \emph{gaussian orthogonal ensemble} (GOE), in which $\xi_{ij} \equiv N(0,1)_\R$ are real gaussians with mean zero and variance one for $i \neq j$, and $\xi_{ii} \equiv N(0,2)_\R$ are real gaussians with mean zero and variance $2$ for $1 \leq i \leq n$ (thus $\sigma^2=2$ in this case).  These ensembles obey Condition {\condo}.  These ensembles are invariant with respect to conjugation by unitary and orthogonal matrices respectively.
\end{example}

Given an $n \times n$ Hermitian matrix $A$, we will denote its $n$ eigenvalues in increasing order as
$$ \lambda_1(A) \leq \ldots \leq \lambda_n(A),$$
and write $\lambda(A) := (\lambda_1(A),\ldots,\lambda_n(A))$.  We also let $u_1(A),\ldots,u_n(A) \in \C^n$ be an orthonormal basis of eigenvectors of $A$ with $A u_i(A) = \lambda_i(A) u_i(A)$.

We also introduce the \emph{eigenvalue counting function}
\begin{equation}\label{eigencount}
 N_I(A) := | \{ 1 \leq i \leq n: \lambda_i(A) \in I \}|
\end{equation}
for any interval $I \subset \R$.  We will be interested in both the \emph{coarse-scale} eigenvalue counting function $N_I(W_n)$ and the \emph{fine-scale} eigenvalue counting function $N_I(A_n)$.

\section{The local semi-circular law}\label{semi-sec}

The most fundamental result about the spectrum of Wigner matrices is the \emph{Wigner semi-circular law}.  We state here a powerful local version of this law, due to Erd\H os, Schlein, and Yau \cite{ESY1,ESY2,ESY3} (see also \cite{EYY}, \cite{EYY2}, \cite{ekyy}, \cite{ekyy2} for further refinements).   Denote by $\rho_{sc}$ the semi-circle density function with
support on $[-2,2]$,
\begin{equation}\label{semi}
 \rho_{sc} (x):= \frac{1}{2\pi} (4-x^2)_+^{1/2}.
\end{equation}

\begin{theorem}[Local semi-circle law]\label{lsc}  Let $M_n$ be a Wigner matrix obeying Condition \condo, let $\eps > 0$, and let $I \subset \R$ be an interval of length $|I| \geq n^{-1+\eps}$.  Then with overwhelming probability\footnote{By this, we mean that the event occurs with probability $1 - O_A(n^{-A})$ for each $A>0$.}, one has\footnote{We use the asymptotic notation $o(X)$ to denote any quantity that goes to zero as $n \to \infty$ when divided by $X$, and $O(X)$ to denote any quantity bounded in magnitude by $CX$, where $C$ is a constant independent of $n$.}
\begin{equation}\label{nawi}
 N_I(W_n) = n\int_I \rho_{\sc}(x)\ dx + o(n|I|).
\end{equation}
\end{theorem}

\begin{proof} See e.g. \cite[Theorem 1.10]{TVlocal2}.  For the most precise estimates currently known of this type (and with the weakest decay hypotheses on the entries), see \cite{ekyy}.  The proofs are based on the Stieltjes transform method; see e.g. \cite{BSbook} for an exposition of this method.
\end{proof}

A variant of Theorem \ref{lsc}, established subsequently\footnote{The result in \cite{EYY2} actually proves a more precise result that also gives sharp results in the edge of the spectrum, though due to the sparser nature of the $\lambda_i^\cl(W_n)$ in that case, the error term $O_\eps(n^{-1+\eps})$ must be enlarged.} in \cite{EYY2},  is the extremely useful \emph{eigenvalue rigidity property}
\begin{equation}\label{eigenrigid}
 \lambda_i(W_n) = \lambda_i^{\cl}(W_n) + O_\eps(n^{-1+\eps}),
\end{equation}
valid with overwhelming probability in the bulk range $\delta n \leq i \leq (1-\delta) n$ for any fixed $\delta>0$ (and assuming Condition \condo).   This result is key in some of the strongest applications of the theory.  Here the \emph{classical location} $\lambda_i^{\cl}(W_n)$ of the $i^{\operatorname{th}}$ eigenvalue is the element of $[-2,2]$ defined by the formula
$$ \int_{-2}^{\lambda_i^{\cl}(W_n)} \rho_\sc(y)\ dy = \frac{i}{n}.$$

Roughly speaking, results such as Theorem \ref{lsc} and \eqref{eigenrigid} control the spectrum of $W_n$ at scales $n^{-1+\eps}$ and above.  However, they break down at the fine scale $n^{-1}$; indeed, for intervals $I$ of length $|I|=O(1/n)$, one has $n\int_I \rho_{\sc}(x)\ dx = O(1)$, while $N_I(W_n)$ is clearly a natural number, so that one can no longer expect an asymptotic of the form \eqref{nawi}.  Nevertheless, local semicircle laws are an essential part of the fine-scale theory.  One particularly useful consequence of these laws is that of \emph{eigenvector delocalisation}:

\begin{corollary}[Eigenvalue delocalisation]\label{deloc} Let $M_n$ be a Wigner matrix obeying Condition \condo, and let $\eps > 0$.  Then with overwhelming probability, one has $u_i(W_n)^* e_j = O( n^{-1/2 + \eps} )$ for all $1 \leq i,j \leq n$, where the $e_1,\ldots,e_n$ are the standard basis of $\C^n$.
\end{corollary}

Note from Pythagoras' theorem that $\sum_{j=1}^n |u_i(M_n)^* e_j|^2 = \|u_i(M_n)\|^2 = 1$; thus Corollary \ref{deloc} asserts, roughly speaking, that the coefficients of each eigenvector are as spread out (or \emph{delocalised}) as possible.

Corollary \ref{deloc} can be established in a number of ways.  One particularly slick approach proceeds via control of the resolvent (or Green's function) $(W_n - zI)^{-1}$, taking advantage of the identity
$$ \Im ((W_n-zI)^{-1})_{jj} = \sum_{i=1}^n \frac{\eta}{(\lambda_i(W_n)-E)^2 + \eta^2} |u_i(W_n)^* e_j|^2$$
for $z = E+\sqrt{-1}\eta$; it turns out that the machinery used to prove Theorem \ref{lsc} also can be used to control the resolvent.  See for instance \cite{Erd} for details of this approach.  

\section{GUE and gauss divisible ensembles}\label{gue-sec}

We now turn to the question of the fine-scale behavior of eigenvalues of Wigner matrices, starting with the model case of GUE.  Here, it is convenient to work with the fine-scale normalisation $A_n := \sqrt{n} M_n$.  For simplicity we will restrict attention to the bulk region of the spectrum, which in the fine-scale normalisation corresponds to eigenvalues $\lambda_i(A_n)$ of $A_n$ that are near $nu$ for some fixed $-2<u<2$ independent of $n$.

A basic object of study are the \emph{$k$-point correlation functions} $R_n^{(k)} = R_n^{(k)}(A_n): \R^k \to \R^+$, defined via duality to be the unique symmetric function (or measure) for which one has
\begin{equation}\label{rfk}
 \int_{\R^k} F(x_1,\ldots,x_k) R_n^{(k)}(x_1,\ldots,x_k)\ dx_1 \ldots dx_k = k! \sum_{1 \leq i_1 < \ldots < i_k} \E F(\lambda_{i_1}(A_n), \ldots, \lambda_{i_k}(A_n) )
\end{equation}
for all symmetric continuous compactly supported functions $F: \R^k \to \R$.  Alternatively, one can write
$$ R_n^{(k)}(x_1,\ldots,x_k) = \frac{n!}{(n-k)!} \int_{\R^{n-k}} \rho_n(x_1,\ldots,x_n)\ dx_{k+1} \ldots dx_n$$
where $\rho_n := \frac{1}{n!} R_n^{(n)}$ is the symmetrized joint probability distribution of all $n$ eigenvalues of $A_n$.

From the semi-circular law, we expect that at the energy level $nu$ for some $-2 < u < 2$, the eigenvalues of $A_n$ will be spaced with average spacing $1/\rho_\sc(u)$.  It is thus natural to consider the \emph{normalised $k$-point correlation function} $\rho^{(k)}_{n,u} = \rho^{(k)}_{n,u}(A_n): \R^k \to \R^+$, defined by the formula
\begin{equation}\label{roo}
 \rho^{(k)}_{n,u}(x_1,\ldots,x_k) := R_n^{(k)}\left( nu + \frac{x_1}{\rho_\sc(u)}, \ldots, nu + \frac{x_k}{\rho_\sc(u)} \right).
\end{equation}

It has been generally believed (and in many cases explicitly conjectured; see e.g. \cite[page 9]{Meh}) that the asymptotic statistics for the quantities mentioned above are \emph{universal}, in the sense that the limiting laws do not depend on the distribution of the atom variables (assuming of course that they have been normalised as stated in Definition \ref{def:Wignermatrix}).  This phenomenon was motivated by examples of similarly universal laws in physics, such as the laws of thermodynamics or of critical percolation; see e.g. \cite{Meh, Deibook, Deisur} for further discussion.

It is clear that if one is able to prove the universality of a limiting law, then it suffices to 
compute this law for one specific model in order to describe the asymptotic behaviour for all other models. A natural choice for the specific model is GUE, as for this model, many 
limiting laws can be computed directly thanks to the availability of an explicit formula for the joint distribution of the eigenvalues, as well as the useful identities of determinantal processes.  For instance, one has \emph{Ginibre's formula}
\begin{equation}\label{rhun}
\rho_n(x_1,\ldots,x_n) = \frac{1}{(2\pi n)^{n/2}} e^{-|x|^2/2n} \prod_{1 \leq i<j \leq n} |x_i-x_j|^2,
\end{equation}
for the joint eigenvalue distribution, as can be verified from a standard calculation; see \cite{gin}.  From this formula, the theory of determinantal processes, and asymptotics for Hermite polynomials, one can then obtain the limiting law
\begin{equation}\label{k-asym}
 \lim_{n \to \infty} \rho^{(k)}_{n,u}(x_1,\ldots,x_k) = \rho^{(k)}_{\Dyson}(x_1,\ldots,x_k)
\end{equation}
locally uniformly in $x_1,\ldots,x_k$ where
$$ \rho^{(k)}_{\Dyson}(x_1,\ldots,x_k) := \det( K_{\Dyson}(x_i,x_j) )_{1 \leq i,j \leq k}$$
and $K_{\Dyson}$ is the \emph{Dyson sine kernel}
$$ K_{\Dyson}(x,y) := \frac{\sin(\pi(x-y))}{\pi(x-y)}$$
(with the usual convention that $\frac{\sin x}{x}$ equals $1$ at the origin); see \cite{gin,Meh}.  

Using a general central limit theorem for determinantal processes due to Costin-Leibowitz \cite{CLe} and Soshnikov \cite{Sos2}, one can then give a limiting law for $N_I(A_n)$ in the case of the macroscopic intervals $I=[nu,+\infty)$.  More precisely, one has the central limit theorem
$$ \frac{N_{[nu,+\infty)}(A_n) - n \int_u^\infty \rho_\sc(y)\ dy}{\sqrt{\frac{1}{2\pi^2} \log n}} \to N(0,1)_\R$$
in the sense of probability distributions, for any $-2 < u < 2$; see \cite{Gus}.  By using the counting functions $N_{[nu,+\infty)}$ to solve for the location of individual eigenvalues $\lambda_i(A_n)$, one can then conclude the central limit theorem
\begin{equation}\label{gustav}
 \frac{\lambda_i(A_n) - \lambda_i^\cl(A_n)}{\sqrt{\log n/2\pi} / \rho_\sc(u)} \to N(0,1)_\R
 \end{equation}
whenever $\lambda_i^\cl(A_n) := n \lambda_i^\cl(W_n)$ is equal to $n(u+o(1))$ for some fixed $-2 < u < 2$; see \cite{Gus}.  

The above analysis extends to many other classes of invariant ensembles (such as GOE), for which the joint eigenvalue distribution has a form similar to \eqref{rhun}; see \cite{Deibook} for further discussion.  Another important extension of the above results is to the \emph{gauss divisible} ensembles, which are Wigner matrices $M_t$ of the form
$$ M_n^t = e^{-t/2} M_n^0 + (1-e^{-t})^{1/2} G_n,$$
where $G_n$ is a GUE matrix independent of $M_n^0$.  In particular, the random matrix $M_n^t$ is distributed as $M_n^0$ for $t=0$ and then continuously deforms towards the GUE distribution as $t \to +\infty$.  By using explicit formulae for the eigenvalue distribution of a gauss divisible matrix,  Johansson \cite{Joh1} was able\footnote{Some additional technical hypotheses were assumed in \cite{Joh1}, namely that the diagonal variance $\sigma^2$ was equal to $1$, that the real and imaginary parts of each entry of $M'_n$ were independent, and that the matrix entries had bounded $C_0^{\operatorname{th}}$ moment for some $C_0>6$.} to extend the asymptotic \eqref{k-asym} for the $k$-point correlation function from GUE to the more general class of gauss divisible matrices with fixed parameter $t>0$ (independent of $n$).

It is of interest to extend this analysis to as small a value of $t$ as possible, since if one could set $t=0$ then one would obtain universality for all Wigner ensembles.  By optimising Johansson's method (and taking advantage of the local semi-circle law), Erd\H os, Peche, Ramirez, Schlein, and Yau \cite{EPRSY} was able to extend the universality of \eqref{k-asym} (interpreted in a suitably weak convergence topology, such as vague convergence) to gauss divisible ensembles for $t$ as small as $n^{-1+\eps}$ for any fixed $\eps>0$.

An important alternate approach to these results was developed by Erd\H os, Ramirez, Schlein, Yau, and Yin \cite{ERSY}, \cite{ESY}, \cite{ESYY}, based on a stability analysis of the Dyson Brownian motion \cite{Dys} governing the evolution of the eigenvalues of a matrix Ornstein-Uhlenbeck process.  We refer to \cite{Erd} for a discussion of this method.  Among other things, this argument reproves a weaker version of the result in \cite{EPRSY} mentioned earlier, in which one obtained universality for the asymptotic \eqref{k-asym} after an additional averaging in the energy parameter $u$.  However, the method was simpler and more flexible than that in \cite{EPRSY}, as it did not rely on explicit identities, and has since been extended to many other types of ensembles, including the real symmetric analogue of gauss divisible ensembles in which the role of GUE is replaced instead by GOE.

\section{The Four Moment Theorem}\label{swap-sec}

The results discussed above for invariant or gauss divisible ensembles can be extended to more general Wigner ensembles via a powerful swapping method known as the \emph{Lindeberg exchange strategy}, introduced in Lindeberg's classic proof \cite{lindeberg} of the central limit theorem, and first applied to Wigner ensembles in \cite{Chat}.  This method can be used to control expressions such as $\E F(M_n) - F(M'_n)$, where $M_n, M'_n$ are two (independent) Wigner matrices.  If one can obtain bounds such as
$$\E F(M_n) - \E F(\tilde M_n) = o(1/n)$$
when $\tilde M_n$ is formed from $M_n$ by replacing\footnote{Technically, the matrices $\tilde M_n$ formed by such a swapping procedure are not Wigner matrices as defined in Definition \ref{def:Wignermatrix}, because the diagonal or upper-triangular entries are no longer identically distributed.  However, all of the relevant estimates for Wigner matrices can be extended to the non-identically-distributed case at the cost of making the notation slightly more complicated.  As this is a relatively minor issue, we will not discuss it further here.} one of the diagonal entries $\xi_{ii}$ of $M_n$ by the corresponding entry $\xi'_{ii}$ of $M'_n$, and bounds such as
$$\E F(M_n) - \E F(\tilde M_n) = o(1/n^2)$$
when $\tilde M_n$ is formed from $M_n$ by replacing one of the off-diagonal entries $\xi_{ij}$ of $M_n$ with the corresponding entry $\xi'_{ij}$ of $M'_n$ (and also replacing $\xi_{ji} = \overline{\xi_{ij}}$ with $\xi'_{ji} = \overline{\xi'_{ij}}$, to preserve the Hermitian property), then on summing an appropriate telescoping series, one would be able to conclude asymptotic agreement of the statistics $\E F(M_n)$ and $\E F(M'_n)$:
\begin{equation}\label{famn}
\E F(M_n) - \E F(M'_n) = o(1)
\end{equation}

The \emph{Four Moment Theorem} asserts, roughly speaking, that we can obtain conclusions of the form \eqref{famn} for suitable statistics $F$ as long as $M_n, M'_n$ match to fourth order.  More precisely, we have

\begin{definition}[Matching moments]  Let $k \geq 1$.  Two complex random variables $\xi, \xi'$ are said to \emph{match to order $k$} if one has $\E \Re(\xi)^a \Im(\xi)^b = \E \Re(\xi')^a \Im(\xi')^b$ whenever $a, b \geq 0$ are integers such that $a+b \leq k$.
\end{definition}

In the model case when the real and imaginary parts of $\xi$ or of $\xi'$ are independent, the matching moment condition simplifies to the assertion that $\E \Re(\xi)^a = \E \Re(\xi')^a$ and $\E \Im(\xi)^b = \E \Im(\xi')^b$ for all $0 \leq a, b \leq k$.
 
\begin{theorem}[Four Moment Theorem]\label{theorem:Four} 
Let $c_0 > 0$ be a sufficiently small constant.
 Let $M_n = (\xi_{ij})_{1 \leq i,j \leq n}$ and $M'_n = (\xi'_{ij})_{1 \leq i,j \leq n}$ be
 two Wigner matrices obeying Condition \condo. Assume furthermore that for any $1 \le  i<j \le n$, $\xi_{ij}$ and
 $\xi'_{ij}$  match to order $4$ 
  and for any $1 \le i \le n$, $\xi_{ii}$ and $\xi'_{ii}$ match  to order $2$.  Set $A_n := \sqrt{n} M_n$ and $A'_n := \sqrt{n} M'_n$, let $1 \leq k \leq n^{c_0}$ be an integer,  and let $G: \R^k \to \R$ be a smooth function obeying the derivative bounds
\begin{equation}\label{G-deriv}
|\nabla^j G(x)| \leq n^{c_0}
\end{equation}
for all $0 \leq j \leq 5$ and $x \in \R^k$.
 Then for any $1 \le i_1 < i_2 \dots < i_k \le n$, and for $n$ sufficiently large we have
\begin{equation} \label{eqn:approximation}
 |\E ( G(\lambda_{i_1}(A_n), \dots, \lambda_{i_k}(A_n))) -
 \E ( G(\lambda_{i_1}(A'_n), \dots, \lambda_{i_k}(A'_n)))| \le n^{-c_0}.
\end{equation}
\end{theorem}

A preliminary version of Theorem \ref{theorem:Four} was first established by the authors in \cite{TVlocal1}, in the case\footnote{In the paper, $k$ was held fixed, but an inspection of the argument reveals that it extends without difficulty to the case when $k$ is as large as $n^{c_0}$, for $c_0$ small enough.} of bulk eigenvalues (thus $\delta n \leq i_1,\ldots,i_k \leq (1-\delta) n$ for some absolute constant $\delta > 0$).  In \cite{TVlocal2}, the restriction to the bulk was removed; and in \cite{TVlocal3}, Condition {\condo} was relaxed to a finite moment condition.  We will discuss the proof of this theorem in Section \ref{sketch}.  There is strong evidence that the condition of four matching moments is necessary to obtain the conclusion \eqref{eqn:approximation}; see \cite{TVnec}.

A key technical result used in the proof of the Four Moment Theorem, which is also of independent interest, is the \emph{gap theorem}:

\begin{theorem}[Gap theorem]\label{gap}  Let $M_n$ be a Wigner matrix obeying Condition \condo.  Then for every $c_0>0$ there exists a $c_1>0$ (depending only on $c_0$) such that
$$ \P( |\lambda_{i+1}(A_n) - \lambda_i(A_n)| \leq n^{-c_0} ) \ll n^{-c_1}$$
for all $1 \leq i < n$.
\end{theorem}

For reasons of space we will not discuss the proof of this theorem here, but refer the reader to \cite{TVlocal1}, \cite{TVlocal3}.  Among other things, the gap theorem tells us that eigenvalues of a Wigner matrix are usually simple.  Closely related \emph{level repulsion} estimates were established (under an additional smoothness hypothesis on the atom distributions) in \cite{ESY3}.

Another variant of the Four Moment Theorem was subsequently introduced in \cite{EYY}, in which the eigenvalues $\lambda_{i_j}(A_n)$ appearing in Theorem \ref{theorem:Four} were replaced by the components of the resolvent (or Green's function) $(W_n-z)^{-1}$, but with slightly different technical hypotheses on the matrices $M_n, M'_n$; see \cite{EYY} for full details.  As the resolvent-based quantities are averaged statistics that sum over many eigenvalues, they are far less sensitive to the eigenvalue repulsion phenomenon than the individual eigenvalues, and as such the version of the Four Moment Theorem for Green's function has a somewhat simpler proof (based on resolvent expansions rather than the Hadamard variation formulae and Taylor expansion).  Conversely, though, to use the Four Moment Theorem for Green's function to control individual eigenvalues, while possible, requires a significant amount of additional argument; see \cite{knowles}.  Finally, we remark that the Four Moment Theorem has also been extended to cover eigenvectors as well as eigenvalues; see \cite{TV-vector}, \cite{knowles} for details.

\section{Sketch of proof of four moment theorem}\label{sketch}

In this section we discuss the proof of Theorem \ref{theorem:Four}, following the arguments that originated in \cite{TVlocal1} and refined in \cite{TVlocal3}.  

In addition to Theorem \ref{gap}, a key ingredient is the following truncated version of the Four Moment Theorem, in which one removes the event that two consecutive eigenvalues are too close to each other.  For technical reasons, we need to introduce quantities
$$ Q_i(A_n) := \sum_{j \neq i} \frac{1}{|\lambda_j(A_n) - \lambda_i(A_n)|^2} $$
for $i=1,\ldots,n$, which is a regularised measure of extent to which $\lambda_i(A_n)$ is close to any other eigenvalue of $A_n$.

\begin{theorem}[Truncated Four Moment Theorem]\label{trunc}
Let $c_0 > 0$ be a sufficiently small constant.
 Let $M_n = (\xi_{ij})_{1 \leq i,j \leq n}$ and $M'_n = (\xi'_{ij})_{1 \leq i,j \leq n}$ be
 two Wigner matrices obeying Condition \condo. Assume furthermore that for any $1 \le  i<j \le n$, $\xi_{ij}$ and
 $\xi'_{ij}$  match to order $4$ 
  and for any $1 \le i \le n$, $\xi_{ii}$ and $\xi'_{ii}$ match  to order $2$.  Set $A_n := \sqrt{n} M_n$ and $A'_n := \sqrt{n} M'_n$, let $1 \leq k \leq n^{c_0}$ be an integer, and let 
$$ G = G(\lambda_{i_1},\ldots,\lambda_{i_k},Q_{i_1},\ldots,Q_{i_k})$$
be a smooth function from $\R^k \times \R^k_+$ to $\R$ that is supported in the region
\begin{equation}\label{qii}
 Q_{i_1},\ldots,Q_{i_k} \leq n^{c_0}
\end{equation}
and obeys the derivative bounds
\begin{equation}\label{geo}
 |\nabla^j G(\lambda_{i_1},\ldots,\lambda_{i_k},Q_{i_1},\ldots,Q_{i_k})| \leq n^{c_0}
\end{equation}
for all $0 \leq j \leq 5$.  Then
\begin{equation}\label{beg}
\begin{split}
& \E G( \lambda_{i_1}(A_n),\ldots,\lambda_{i_k}(A_n),Q_{i_1}(A_n),\ldots,Q_{i_k}(A_n)) =\\
&\quad 
  \E G( \lambda_{i_1}(A'_n),\ldots,\lambda_{i_k}(A'_n),Q_{i_1}(A'_n),\ldots,Q_{i_k}(A'_n)) + O(n^{-1/2+O(c_0)}.
\end{split}
 \end{equation}
\end{theorem}

We will discuss the proof of this theorem shortly.  Using Theorem \ref{gap}, one can then deduce Theorem \ref{theorem:Four} from Theorem \ref{trunc} by smoothly truncating in the $Q$ variables: see \cite[\S 3.3]{TVlocal1}.

It remains to establish Theorem \ref{trunc}.  To simplify the exposition slightly, let us assume that the matrices $M_n, M'_n$ are real symmetric rather than Hermitian.  

As indicated in Section \ref{swap-sec}, the basic idea is to use the Lindeberg exchange strategy.  To illustrate the idea, let $\tilde M_n$ be the matrix formed from $M_n$ by replacing a single entry $\xi_{pq}$ of $M_n$ with the corresponding entry $\xi'_{pq}$ of $M'_n$ for some $p<q$, with a similar swap also being performed at the $\xi_{qp}$ entry to keep $\tilde M_n$ Hermitian.  Strictly speaking, $\tilde M_n$ is not a Wigner matrix as defined in Definition \ref{def:Wignermatrix}, as the entries are no longer identically distributed, but this will not significantly affect the arguments.  (One also needs to perform swaps on the diagonal, but this can be handled in essentially the same manner.)

Set $\tilde A_n := \sqrt{n} \tilde M_n$ as usual.  
We will sketch the proof of the claim that
\begin{align*}
& \E G( \lambda_{i_1}(A_n),\ldots,\lambda_{i_k}(A_n),Q_{i_1}(A_n),\ldots,Q_{i_k}(A_n)) \\
&\quad = 
  \E G( \lambda_{i_1}(\tilde A_n),\ldots,\lambda_{i_k}(\tilde A_n),Q_{i_1}(\tilde A_n),\ldots,Q_{i_k}(\tilde A_n)) + O(n^{-5/2+O(c_0)};
  \end{align*}
by telescoping together $O(n^2)$ estimates of this sort one can establish \eqref{beg}.  (For swaps on the diagonal, one only needs an error term of $O(n^{-3/2+O(c_0)})$, since there are only $O(n)$ swaps to be made here rather than $O(n^2)$.  This is ultimately why there are two fewer moment conditions on the diagonal than off it.)

We can write $A_n = A(\xi_{pq})$, $\tilde A_n = A(\xi'_{pq})$, where 
$$
A(t) = A(0) + t A'(0)
$$
is a (random) Hermitian matrix depending linearly\footnote{If we were working with Hermitian matrices rather than real symmetric matrices, then one could either swap the real and imaginary parts of the $\xi_{ij}$ separately (exploiting the hypotheses that these parts were independent), or else repeat the above analysis with $t$ now being a complex parameter (or equivalently, two real parameters) rather than a real one.  In the latter case, one needs to replace all instances of single variable calculus below (such as Taylor expansion) with double variable calculus, but aside from notational difficulties, it is a routine matter to perform this modification.}  on a real parameter $t$, with $A(0)$ being a Wigner matrix with one entry (and its adjoint) zeroed out, and $A'(0)$ is the explicit elementary Hermitian matrix 
\begin{equation}\label{apo}
A'(0) = e_p e_q^* + e_p^* e_q.
\end{equation}
We note the crucial fact that the random matrix $A(0)$ is independent of both $\xi_{pq}$ and $\xi'_{pq}$.  Note from Condition {\condo} that we expect $\xi_{pq}, \xi'_{pq}$ to have size $O(n^{O(c_0)})$ most of the time, so we should (heuristically at least) be able to restrict attention to the regime $t = O(n^{O(c_0)})$.  If we then set
\begin{equation}\label{ftt}
 F(t) := \E G( \lambda_{i_1}(A(t)),\ldots,\lambda_{i_k}(A(t)),Q_{i_1}(A(t)),\ldots,Q_{i_k}(A(t))) 
\end{equation}
then our task is to show that
\begin{equation}\label{fxij}
 \E F(\xi_{pq}) = \E F(\xi'_{pq}) + O(n^{-5/2+O(c_0)}).
\end{equation}

Suppose that we have Taylor expansions of the form
\begin{equation}\label{lait}
 \lambda_{i_l}(A(t)) = \lambda_{i_l}(A(0)) + \sum_{j=1}^4 c_{l,j} t^j + O( n^{-5/2+O(c_0)} )
\end{equation}
for all $t = O(n^{O(c_0)})$ and $l=1,\ldots,k$, where the Taylor coefficients $c_{l,j}$ have size $c_{l,j} = O(n^{-j/2+O(c_0)}$, and similarly for the quantities $Q_{i_l}(A(t))$.  Then by using the hypothesis \eqref{geo} and further Taylor expansion, we can obtain a Taylor expansion
$$
F(t) = F(0) + \sum_{j=1}^4 f_j t^j + O( n^{-5/2+O(c_0)} )
$$
for the function $F(t)$ defined in \eqref{ftt}, where the Taylor coefficients $f_j$ have size $f_j = O(n^{-j/2+O(c_0)})$.  Setting $t$ equal to $\xi_{pq}$ and taking expectations, and noting that the Taylor coefficients $f_j$ depend only on $F$ and $A(0)$ and is thus independent of $\xi_{ij}$, we conclude that
$$ \E F(\xi_{pq}) = \E F(0) + \sum_{j=1}^4 (\E f_j) (\E \xi_{pq}^j) + O( n^{-5/2+O(c_0)} )$$
and similarly for $\E F(\xi'_{pq})$.  If $\xi_{pq}$ and $\xi'_{pq}$ have matching moments to fourth order, this gives \eqref{fxij}.  

It remains to establish \eqref{lait} (as well as the analogue for $Q_{i_l}(A(t))$, which turns out to be analogous).  We abbreviate $i_l$ simply as $i$.  By Taylor's theorem with remainder, it would suffice to show that 
\begin{equation}\label{lat}
 \frac{d^j}{dt^j} \lambda_i(A(t)) = O( n^{-j/2+O(c_0)} )
\end{equation}
for $j=1,\ldots,5$. As it turns out, this is not quite true as stated, but it becomes true (with overwhelming probability\footnote{Technically, each value of $t$ has a different exceptional event of very small probability for which the estimates fail.  Since there are uncountably many values of $t$, this could potentially cause a problem when applying the union bound.  In practice, though, it turns out that one can restrict $t$ to a discrete set, such as the multiples of $n^{-100}$, in which case the union bound can be applied without difficulty.  See \cite{TVlocal1} for details.}) if one can assume that $Q_i(A(t))$ is bounded by $n^{O(c_0)}$.  In principle, one can reduce to this case due to the restriction \eqref{qii} on the support of $G$, although there is a technical issue because one will need to establish the bounds \eqref{lat} for values of $t$ other than $\xi_{pq}$ or $\tilde \xi_{pq}$.  This difficulty can be overcome by a continuity argument; see \cite{TVlocal1}.  For the purposes of this informal discussion, we shall ignore this issue and simply assume that we may restrict to the case where
\begin{equation}\label{qiat}
Q_i(A(t)) \ll n^{O(c_0)}.
\end{equation}
In particular, the eigenvalue $\lambda_i(A(t))$ is simple, which ensures that all quantities depend smoothly on $t$ (locally, at least).

To prove \eqref{lat}, one can use the classical \emph{Hadamard variation formulae} for the derivatives of $\lambda_i(A(t))$, which can be derived for instance by repeatedly differentiating the eigenvector equation $A(t) u_i(A(t)) = \lambda_i(A(t)) u_i(A(t))$.  The formula for the first derivative is
$$ \frac{d}{dt} \lambda_i(A(t)) = u_i(A(t))^* A'(0) u_i(A(t)).$$
But recall from eigenvalue delocalisation (Corollary \ref{deloc}) that with overwhelming probability, all coefficients of $u_i(A(t))$ have size $O(n^{-1/2+o(1)})$; given the nature
of the matrix \eqref{apo}, we can then obtain \eqref{lat} in the $j=1$ case.

Now consider the $j=2$ case.  The second derivative formula reads
$$ \frac{d^2}{dt^2} \lambda_i(A(t)) = - 2 \sum_{j \neq i} \frac{|u_i(A(t))^* A'(0) u_j(A(t))|^2}{\lambda_j(A(t)) - \lambda_i(A(t))}.$$
Using eigenvalue delocalisation as before, we see with overwhelming probability that the numerator is $O(n^{-1+o(1)})$.  To deal with the denominator, one has to exploit the hypothesis \eqref{qiat} and the local semicircle law (Theorem \ref{lsc}).  Using these tools, one can conclude \eqref{lat} in the $j=2$ case with overwhelming probability.

It turns out that one can continue this process for higher values of $j$, although the formulae for the derivatives for $\lambda_i(A(t))$ (and related quantities, such as $P_i(A(t))$ and $Q_i(A(t))$) become increasingly complicated, being given by a certain recursive formula in $j$.  See \cite{TVlocal1} for details.

\section{Distribution of individual eigenvalues}\label{dist-sec}

One of the simplest applications of the above machinery is to extend the central limit theorem \eqref{gustav} of Gustavsson \cite{Gus} for eigenvalues $\lambda_i(A_n)$ in the bulk from GUE to more general ensembles:

\begin{theorem}\label{gust}  The gaussian fluctuation law \eqref{gustav} continues to hold for Wigner matrices obeying Condition \condo, and whose atom distributions match that of GUE to second order on the diagonal and fourth order off the diagonal; thus, one has
$$
 \frac{\lambda_i(A_n) - \lambda_i^\cl(A_n)}{\sqrt{\log n/2\pi} / \rho_\sc(u)} \to N(0,1)_\R
$$ 
whenever $\lambda_i^\cl(A_n)=n(u+o(1))$ for some fixed $-2 < u < 2$.
\end{theorem}

\begin{proof}  Let $M'_n$ be drawn from GUE, thus by \eqref{gustav} one already has
$$
 \frac{\lambda_i(A'_n) - \lambda_i^\cl(A_n)}{\sqrt{\log n/2\pi} / \rho_\sc(u)} \to N(0,1)_\R
$$ 
(note that $\lambda_i^\cl(A_n) = \lambda_i^\cl(A'_n)$.  To conclude the analogous claim for $A_n$, it suffices to show that
\begin{equation}\label{lama}
\P( \lambda_i( A'_n ) \in I_- ) - n^{-c_0} \leq
\P( \lambda_i( A_n ) \in I ) \leq \P( \lambda_i( A'_n ) \in I_+ ) + n^{-c_0}
\end{equation}
for all intervals $I = [a,b]$, and $n$ sufficiently large, where $I_+ := [a-n^{-c_0/10}, b+n^{-c_0/10}]$ and $I_- := [a+n^{-c_0/10},b-n^{-c_0/10}]$.

We will just prove the second inequality in \eqref{lama}, as the first is very similar.
We define a smooth bump function $G:\R \to \R^+$ equal to one on $I_-$ and vanishing outside of $I_+$.  Then we have
$$ \P( \lambda_i( A_n ) \in I ) \leq \E G( \lambda_i(A_n) ) $$
and
$$ \E G( \lambda_i(A'_n) ) \leq \P( \lambda_i( A'_n ) \in I )$$
On the other hand, one can choose $G$ to obey \eqref{G-deriv}.  Thus by Theorem \ref{theorem:Four} we have
$$ |\E G( \lambda_i(A_n) ) - \E G( \lambda_i(A'_n) )| \leq n^{-c_0}$$
and the second inequality in \eqref{lama} follows from the triangle inequality.  The first inequality is similarly proven using a smooth function that equals $1$ on $I_-$ and vanishes outside of $I$.
\end{proof}

\begin{remark}  In \cite{Gus} the asymptotic joint distribution of $k$ distinct eigenvalues $\lambda_{i_1}(M_n),\ldots,\lambda_{i_k}(M_n)$ in the bulk of a GUE matrix $M_n$ was computed (it is a gaussian $k$-tuple with an explicit covariance matrix).  By using the above argument, one can extend that asymptotic for any fixed $k$ to other Wigner matrices, so long as they match GUE to fourth order off the diagonal and to second order on the diagonal.

If one could extend the results in \cite{Gus} to broader ensembles of matrices, such as gauss divisible matrices, then the above argument would allow some of the moment matching hypotheses to be dropped, using tools such as Lemma \ref{match}.
\end{remark}

\begin{remark} Recently in \cite{Doering}, a moderate deviations property of the distribution of the eigenvalues $\lambda_i(A_n)$ was established first for GUE, and then extended to the same class of matrices considered in Theorem \ref{gust} by using the Four Moment Theorem.  An analogue of Theorem \ref{gust} for real symmetric matrices (using GOE instead of GUE) was established in \cite{rourke}.
\end{remark}

There are similar results at the edge of the spectrum, though with several additional technicalities; see \cite{Sos1,  Ruz, Khor, TVlocal2, Joh2, EYY2}.

\section{The Wigner-Dyson-Mehta conjecture}

We now consider the extent to which the asymptotic \eqref{k-asym}, which asserts that the normalised $k$-point correlation functions $\rho_{n,u}^{(k)}$ converge to the universal limit $\rho^{(k)}_\Dyson$, can be extended to more general Wigner ensembles.  A long-standing conjecture of Wigner, Dyson, and Mehta (see e.g. \cite{Meh}) asserts (informally speaking) that \eqref{k-asym} is valid for all fixed $k$, all Wigner matrices and all fixed energy levels $-2 < u < 2$ in the bulk.  However, to make this conjecture precise one has to specify the nature of convergence in \eqref{k-asym}.  For GUE, the convergence is quite strong (in the local uniform sense), but one cannot expect such strong convergence in general, particularly in the case of discrete ensembles in which $\rho_{n,u}^{(k)}$ is a discrete probability distribution (i.e. a linear combination of Dirac masses) and thus is unable to converge uniformly or pointwise to the continuous limiting distribution $\rho^{(k)}_\Dyson$.  We will thus instead settle for the weaker notion of \emph{vague convergence}.  More precisely, we say that \eqref{k-asym} holds in the vague sense if one has
\begin{equation}\label{test}
 \int_{\R^k} F(x_1,\ldots,x_k) \rho^{(k)}_{n,u}(x_1,\ldots,x_k)\ dx_1 \ldots dx_k = \int_{\R^k} F(x_1,\ldots,x_k) \rho^{(k)}_{\Dyson}(x_1,\ldots,x_k)\ dx_1 \ldots dx_k
 \end{equation}
for all continuous, compactly supported functions $F: \R^k \to \R$.  By the Stone-Weierstrass theorem we may take $F$ to be a test function (i.e. smooth and compactly supported) without loss of generality.  

The Wigner-Dyson-Mehta conjecture is largely resolved in the vague convergence category:

\begin{theorem}[Wigner-Dyson-Mehta conjecture in the vague sense]\label{vague-o}   Let $M_n$ be a Wigner matrix obeying Condition \condo, and let $-2 < u < 2$ and $k \geq 1$ be fixed.  Then \eqref{k-asym} holds in the vague sense.
\end{theorem}

This theorem, proven in \cite{TVmeh}, builds upon a long sequence of partial results towards the Wigner-Dyson-Mehta conjecture \cite{Joh1, ERSY, EPRSY, TVlocal1, ERSTVY, EYY, EYY2}, which we will summarise (in a slightly non-chronological order) below.  As recalled in Section \ref{gue-sec}, the asymptotic \eqref{k-asym} for GUE (in the sense of locally uniform convergence, which is far stronger than vague convergence) follows as a consequence of the Gaudin-Mehta formula and the Plancherel-Rotach asymptotics for Hermite polynomials\footnote{Analogous results are known for much wider classes of invariant random matrix ensembles, see e.g. \cite{DKMVZ}, \cite{PS}, \cite{BI}.  However, we will not discuss these results further here, as they do not directly impact on the case of Wigner ensembles.}.  

The next major breakthrough was by Johansson \cite{Joh1}, who, as discussed previously, established \eqref{k-asym} for gauss divisible ensembles at some fixed time parameter $t>0$ independent of $n$, obtained \eqref{k-asym} in the vague sense (in fact, the slightly stronger convergence of \emph{weak convergence} was established in that paper, in which the function $F$ in \eqref{test} was allowed to merely be $L^\infty$ and compactly supported, rather than continuous and compactly supported).    The main tool used in \cite{Joh1} was an explicit determinantal formula for the correlation functions in the gauss divisible case, essentially due to Br\'ezin and Hikami \cite{brezin}.

In Johansson's result, the time parameter $t > 0$ had to be independent of $n$.  It was realized by Erd\H{o}s, Ramirez, Schlein, and Yau that one could obtain many further cases of the Wigner-Dyson-Mehta conjecture if one could extend Johansson's result to much shorter times $t$ that decayed at a polynomial rate in $n$.  This was first achieved (again in the context of weak convergence) for $t > n^{-3/4+\eps}$ for an arbitrary fixed $\eps>0$ in \cite{ERSY}, and then to the essentially optimal case $t > n^{-1+\eps}$ (for weak convergence, and (implicitly) in the local $L^1$ sense as well) in \cite{EPRSY}.  By combining this with the method of reverse heat flow discussed in Section \ref{swap-sec}, the asymptotic \eqref{k-asym} (again in the sense of weak convergence) was established for all Wigner matrices whose distribution obeyed certain smoothness conditions (e.g. when $k=2$ one needs a $C^6$ type condition), and also decayed exponentially.  The methods used in \cite{EPRSY} were an extension of those in \cite{Joh1}, combined with an approximation argument (the ``method of time reversal'') that approximated a continuous distribution by a gauss divisible one (with a small value of $t$); the arguments in \cite{ERSY} are based instead on an analysis of the Dyson Brownian motion.

By combining the above observation with the moment matching lemma presented below, one immediately concludes Theorem \ref{vague-o} assuming that the off-diagonal atom distributions are supported on at least three points.

\begin{lemma}[Moment matching lemma]\label{match}  Let $\xi$ be a real random variable with mean zero, variance one, finite fourth moment, and which is supported on at least three points.  Then there exists a gauss divisible, exponentially decaying real random variable $\xi'$ that matches $\xi$ to fourth order.
\end{lemma}

For a proof of this lemma, see \cite[Lemma 28]{TVlocal1}.  The requirement of support on at least three points is necessary; indeed, if $\xi$ is supported in just two points $a,b$, then $\E (\xi-a)^2 (\xi-b)^2 = 0$, and so any other distribution that matches $\xi$ to fourth order must also be supported on $a,b$ and thus cannot be gauss divisible.  

To remove the requirement that the atom distributions be supported on at least three points, one can observe from the proof of the four moment theorem that one only needs the moments of $M_n$ and $M'_n$ to \emph{approximately} match to fourth order in order to be able to transfer results on the distribution of spectra of $M_n$ to that of $M'_n$.  In particular, if $t = n^{-1+\eps}$ for some small $\eps > 0$, then the gauss divisible matrix $M^t_n$ associated to $M_n$ at time $t$ is already close enough to matching the first four moments of $M_n$ to apply (a version of) the Four Moment Theorem.  The results of \cite{EPRSY} give the asymptotic \eqref{k-asym} for $M^t_n$, and the eigenvalue rigidity property \eqref{eigenrigid} then allows one to transfer this property to $M_n$, giving Theorem \ref{vague-o}.

\begin{remark}\label{chron}  The above presentation (drawn from the most recent paper \cite{TVmeh}) is somewhat ahistorical, as the arguments used above emerged from a sequence of papers, which obtained partial results using the best technology available at the time.  In the paper \cite{TVlocal1}, where the first version of the Four Moment Theorem was introduced, the asymptotic \eqref{k-asym} was established under the additional assumptions of Condition \condo, and matching the GUE to fourth order; the former hypothesis was due to the weaker form of the four moment theorem known at the time, and the latter was due to the fact that the eigenvalue rigidity result \eqref{eigenrigid} was not yet established (and was instead deduced from the results of Gustavsson \cite{Gus} combined with the Four Moment Theorem, thus necessitating the matching moment hypothesis).  For related reasons, the paper in \cite{ERSTVY} (which first introduced the use of an approximate Four Moment Theorem) was only able to establish \eqref{k-asym} after an additional averaging in the energy parameter $u$ (and with Condition \condo).  The subsequent progress in \cite{ESY} via heat flow methods gave an alternate approach to establishing \eqref{k-asym}, but also required an averaging in the energy and a hypothesis that the atom distributions be supported on at least three points, although the latter condition was then removed in \cite{EYY2}.  In a very recent paper \cite{ekyy2}, Condition {\condo} has been relaxed to finite $(4+\eps)^{\operatorname{th}}$ moment of the entries for any fixed $\eps>0$, though still at the cost of averaging in the energy parameter.  Some generalisations in other directions (e.g. to covariance matrices, or to generalised Wigner ensembles with non-constant variances) were also established in \cite{BenP}, \cite{TVlocal3}, \cite{ESYY}, \cite{EYY}, \cite{EYY2}, \cite{ekyy}, \cite{ekyy2}, \cite{wang}.
\end{remark}

\begin{remark} While Theorem \ref{vague-o} is the ``right'' result for discrete Wigner ensembles (except for the hypothesis of Condition \condo, which in view of the results in \cite{ekyy2} should be relaxed significantly), one expects stronger notions of convergence when one has more smoothness hypotheses on the atom distribution; in particular, one should have local uniform convergence of the correlation functions when the distribution is smooth enough.  Some very recent progress in this direction in the $k=1$ case was obtained by Maltsev and Schlein \cite{maltsev}, \cite{maltsev2}.
\end{remark}

\end{document}